\documentclass[10pt]{amsart}
\usepackage[]{amsmath, amsthm, amsfonts,amssymb}
\usepackage{graphicx}
\usepackage[all]{xy}
\usepackage[colorlinks=true]{hyperref}

\newcommand {\N} {{\mathbb N}}
\newcommand {\C} {{\mathbb C}}
\newcommand {\R} {{\mathbb R}}

\newcommand {\Z} {{\mathbb Z}}
\newcommand {\Q} {{\mathbb Q}}
\newcommand {\PP} {{\mathbb P}}

\newcommand {\OO} {{\mathcal O}}

\newcommand {\dt} {\bullet}

\newcommand {\V} { \mathcal{V}}

\DeclareMathOperator{\Spec}{Spec}
\DeclareMathOperator{\im }{im}

\newtheorem{thm}[subsection]{Theorem}
\newtheorem{cor}[subsection]{Corollary}
\newtheorem{lemma}[subsection]{Lemma}
\newtheorem{prop}[subsection]{Proposition}

\newtheorem{remark}[subsection]{Remark}

\begin{document}

 \title[Hodge theory of cyclic covers]{Hodge theory of cyclic covers branched over a union of hyperplanes}

\author{ Donu Arapura}

\thanks{Partially supported by NSF}
\address{Department of Mathematics\\
Purdue University\\
West Lafayette, IN 47907\\
U.S.A.}
\begin{abstract}
  Suppose that $Y$ is a cyclic cover of projective space branched over a hyperplane arrangement $D$, and that $U$ is the complement of the ramification locus in $Y$. The first theorem implies that the Beilinson-Hodge conjecture holds for $U$ if certain multiplicities of $D$ are coprime to the degree of the cover. For instance this applies when $D$ is reduced with normal crossings. The second theorem shows that when $D$ has normal crossings and the degree of the cover is a prime number, the generalized Hodge conjecture holds for any toroidal resolution of $Y$. The last section contains some partial extensions to more general nonabelian covers.
\end{abstract}
\maketitle

The principal goal of this  paper is to verify some standard conjectures in Hodge
theory for a natural class of examples.
Fix  integers $d>1$, $m,n\ge 1$, and consider the cyclic cover $\pi:Y\to \PP_\C^n$ defined
by the weighted homogeneous equation $y^d=f(x_0,\ldots, x_n)$, where $f$ is a product of $md$ distinct  linear
forms. Let $D$ be the divisor defined by $f=0$ and let $E=
\pi^{-1}D$. Our first theorem is that the Beilinson-Hodge conjecture,
as formulated in \cite{as}, holds for $U=Y-E$ if for instance $D$ has
normal crossings. This means that all weight $2j$ Hodge cycles
in $H^j(U,\Q)$ lie in the image of the cycle map from motivic
cohomology. 
The key point is to show
that the weight $2j$ Hodge  cycles on $U$ come from $\PP^n-D$. Then the theorem is almost
immediate.   We note that the theorem is valid even in some cases when
$D$ fails to be reduced or  have normal crossings. The
precise  condition is that the multiplicities of the components of $D$, and
their sums at points where $D$ fails to have  normal crossings, should be
coprime to $d$.

For the second result, we   assume that
$D$ is an arrangement of $d$ hyperplanes with  normal crossings (so
that $m=1$). Then $Y$ is a
singular toroidal variety, so we may choose a toroidal desingularization $X\to Y$
\cite{mumford}.  
Our second theorem implies that the generalized Hodge
conjecture \cite{groth} holds for
$X$ when $d$ is prime. 
 Another notable consequence of the  theorem
is that when $n$ is odd and $d$ prime, the maximal abelian subvariety
of the intermediate Jacobian $J^n(X)$ is zero.
The verification  of the Hodge conjecture for
$X$ goes as follows.  We check that all the   Hodge cycles on $X$ are
either algebraic or pull backs of Hodge cycles from $Y$. 
To analyze the cycles in the second
category, we employ a nice  trick used by
Shioda \cite{shioda} in a similar context.
 By  exploiting the action of $\Z/d\Z$
on cohomology,  we obtain  a very strong bound on the dimension of the space of
Hodge cycles on $Y$. When $d$ is prime, it will imply that there are no
transcendental Hodge cycles.

Let us now indicate the nature of the bounds used above. Given a
smooth projective variety $Y$, or  more generally an orbifold  as above, the dimension of the space of Hodge cycles of type
$(p,p)$ on it is of course bounded by the Hodge number $h^{pp}$. This can be
written as the $2p$th Betti number minus $2\dim T$, where $T\subset H^{2p}(Y)$ is the space
of $(a,b)$ classes with $a<b$. If a possibly nonabelian finite group $G=\{1,g_1,\ldots\}$
acts on $Y$, then $2\dim T$ can be replaced by the dimension of  the smallest
rationally defined $G$-module containing $T+\bar T$. In some cases
this yields a huge improvement. For example, if $G$ acts irreducibly
on $T$ with a nonreal character $\chi$. Then the factor in front of $\dim T$
jumps from $2$ to the
product of the degree of the number field $\Q(\chi(g_1),\ldots)$ times the
Schur index of $\chi$.
This follows from  more general results given in the final section of the paper.
The main result in this section is that under suitable conditions,  there is an explicit bound for the
dimension of the space of Hodge cycles on a branched $G$-cover in
terms of branching data.

\section{Preliminaries on cyclic covers}

We start with a slightly more general set up than given in the
introduction. Let $Z$ be an $n$ dimensional smooth projective variety with a
line bundle $L$. Let $D\subset Z$ be a not necessarily reduced effective divisor with simple normal
crossings such that $\OO_Z(D)\cong L^d$. Then we form the normal $d$-fold cyclic
cover 
$$Y= \text{Normalization of }\mathbf{Spec}(
\bigoplus_{i=0}^{d-1} L^{-i})\stackrel{\pi}{\to} Z$$
branched over $D$ (cf. \cite[\S 3]{ev}).  
Let $V= Z-D$, $E=\pi^{-1}D$ and $U=Y-E$. It is convenient to set $\Omega^k_Z(\log
  D)=\Omega_Z^k(\log D_{red})$ below.
The following is probably known to experts, but we do not know of a
good reference.

\begin{lemma}\label{lemma:toric0}
  $Y$ is local analytically isomorphic to a toric variety with finite
  quotient singularities.
\end{lemma}

\begin{proof}
  Local analytically
$Y$ looks like the normalization of an affine variety of
the form
\begin{equation}
  \label{eq:toric1}
y^d = x_1^{a_1}\ldots x_n^{a_n}  
\end{equation}
Let $R = R_{(a_1,\ldots, a_n,d)}$
be the
quotient of $\C[x_1,\ldots, x_n,y]$ by the ideal generated by the
difference of the two sides of this equation, and let $\tilde R_{(a_1,\ldots, a_n,d)}$ denote its normalization.
The most important case for us is when all  the exponents $a_{i}=1$.
In this case,  the lemma is easy to see directly. The ring $R$
is the ring of invariants of $\C[u_1,\ldots, u_n, v]$
under $(\zeta_j)\in (\mu_{d})^{n}$ acting by 
$$u_j\mapsto \zeta_j u_j;\> v\mapsto \prod \zeta_j^{-1}v$$
 Therefore, in this case, $R$ is already normal.
 The fact that it is also toric is immediate from the shape of the
 equation \eqref{eq:toric1} which is the equality of two monomials.

 For the general case, we will use toric methods more explicitly. But
first, we make a series of reductions.
 Let $g= \gcd(a_1,\ldots,
a_n,d)$.
If $g>1$, then $\Spec R$ is reducible and the components are
isomorphic to $\Spec  R_{(a_1/g,\ldots, a_n/g,d/g)}$. Therefore, we may
reduce to the case that $g=1$.

 If some $a_i=0$, then 
$$R= R_{(a_1,\ldots\hat{a_i}, a_n,d)}\otimes \C[x_i]$$
Thus we may assume that all $a_i>0$.
Let $S\subset \Z^n$ be the sub semigroup  generated by 
$v_1=(d,0,\ldots 0), \ldots, v_n= (0,\ldots, d),$ $ v_{n+1}=
(a_1,\ldots, a_n)$, and let $L\supset S$ denote the sublattice
generated by the same vectors.
The semigroup ring $\C[S]$ can be identified with the subring of
$\C[u_1,\ldots, u_n]$ generated by $u_1^{d}, \ldots, u_n^{d}$ and
$u_1^{a_1}\ldots u_n^{a_n}$. 
The  homomorphism 
$$\C[x_1,\ldots, x_n,y]\to \C[S]$$
defined by 
$$x_i\mapsto u_i^d, y\mapsto u_1^{a_1}\ldots u_n^{a_n}$$
gives an isomorphism $R\cong \C[S]$. The normalization $\tilde R$ of $R$
is given  by the semigroup ring $\C[\tilde S]$, where $S\subset \tilde
S\subset L$ is the saturation \cite[Chap 1,\S 3]{cls}, which is
the intersection of $L$ with the real cone 
$$\{\sum r_iv_i\in L\otimes \R\mid r_i\ge 0\} $$
$\tilde S$ is simplicial, since it is generated by the real basis vectors $v_1,\ldots,
v_n$. Therefore $\Spec R$ has quotient singularities \cite[Chap 3, \S 1]{cls}.
\end{proof}

Thus $Y$ is an orbifold, which for our purposes simply means that it
has finite quotient singularities. Note also that  the  singularities lie over the singular locus
$D_{sing}\subset D$. Bailey \cite{bailey}, and later
Steenbrink \cite{steenbrink}, observed that
most of the standard Hodge theoretic statements generalize from smooth
varieties to  orbifolds. We list the
results that we need from the second reference.

\begin{enumerate}
\item[(H1)] The mixed Hodge structure on $H^i(Y)$ is pure of weight $i$.
We can identify $Gr_F^kH^i(Y)$ with $H^{i-k}(Y,\tilde \Omega_Y^k)$
where $\tilde \Omega_Y^k:=(\Omega_Y^k)^{**}= j_*\Omega_{W}^k$
and $j:W\to Y$ is the embedding of the smooth locus.

\item[(H2)] The hard Lefschetz theorem holds.

\item[(H3)] There is an (noncanonical) isomorphism
$$H^i(Y-D,\C)\cong \bigoplus_{a+b=i} H^a(Y,\tilde
\Omega_Y^b(\log D))$$
where $\tilde \Omega_Y^k(\log D)=j_*\Omega_W^k(\log D\cap W)$.
(Although this is not explicitly stated there, it follows from the discussion in \cite[1.17-1.20]{steenbrink}
and  the fact that  the spectral sequence
associated to $(\tilde \Omega_Y^\dt(\log D),
\tilde \Omega_Y^{\ge k}(\log D)$) degenerates at $E_1$ because it is
part of a cohomological mixed Hodge complex  \cite[8.1.9]{deligne}.)

\end{enumerate}

The group of $d$th roots of unity $\mu_d\cong \Z/d$ acts on $Y$, and we will need to
analyze the eigenspaces on cohomology.  Let $\epsilon:\mu_d\to \C^*$ denote the
standard character given by $\epsilon(\zeta)=\zeta$. Any $\C[\mu_d]$-module $T$ can
be  decomposed as a sum $T=\bigoplus_{i=0}^{d-1} T_{\epsilon^i}$ where
$T_{\epsilon^i}$ is the maximal submodule where $\zeta\in \mu_d$ acts
by multiplication by
$\epsilon^i(\zeta)=\zeta^i$. Define  the nontrivial part of $T$ by $T_{nt}=\bigoplus_{i=1}^{d-1} T_{\epsilon^i}$.
Let $H^{jk}_?(Y)=H^k(Y,\tilde \Omega_Y^j)_?$ etc.,  where
$?=\epsilon^i, nt$.

\begin{lemma}[Esnault-Viehweg]\label{lemma:ev}
Let $D=\sum a_jD_j$, where $D_j$ are the irreducible components. 
Let $[rD] = \sum [a_jr] D_j$, $L^{(-i)}=
L^{-i}\left(\left[\frac{i}{d}D\right]\right) $, and $D^{(i)}$ be the sum of components $D_j$ such that $d\not| ia_j$.
Then
  $$H _{\epsilon^i}^j(Y,\tilde \Omega_Y^k)=
\begin{cases}
H^j(\Omega_{Z}^k) & \text{ if } i=0\\
H^j(\Omega_{Z}^k(\log D^{(i)})\otimes L^{(-i)}) & \text{ otherwise}
\end{cases}
$$

 $$H _{\epsilon^i}^j(Y,\tilde \Omega_Y^k(\log E))=
\begin{cases}
H^j(\Omega_{Z}^k(\log D)) & \text{ if } i=0\\
H^j(\Omega_{Z}^k(\log D)\otimes L^{(-i)}) & \text{ otherwise}
\end{cases}
$$
\end{lemma}

\begin{proof}
Since $\pi$ is finite, we have $H^j(Y,\tilde \Omega_Y^k)=H^j(Z,\pi_*\tilde
\Omega_Y^k)$ as $\C[\mu_d]$-modules. Let $W=Z-D_{sing}$.
Esnault and Viehweg \cite[lemma 3.16]{ev} have shown  that
$$(\pi_*\tilde \Omega_Y^k)_{\epsilon^i}|_W = 
\begin{cases}
\Omega_{W}^k & \text{ if } i=0\\
\Omega_{W}^k(\log D^{(i)})\otimes L^{(-i)}& \text{ otherwise}
\end{cases}
$$
Equality extends to  $Z$ because
these sheaves are reflexive.

The second part also follows from [loc. cit] by the same argument.
\end{proof}

\begin{cor}\label{cor:ev}
The invariant part  $H_{\epsilon^0}^*(Y)$ is isomorphic to
$H^*(Z)$.
\end{cor}

\begin{remark}
The above formulas  simplify under the following assumptions
\begin{enumerate}
\item If the coefficients $a_j$ are coprime to $d$ then $D^{(i)}=D_{red}$
  for all $1\le i\le d-1$. This coprimality condition is equivalent to the map $Y\to Z$ being
  totally ramified along $D$. 
\item If $D$ is reduced then additionally $L^{(-i)}=L^{-i}$
\end{enumerate}
\end{remark}

The following is a special case of much more general vanishing
theorems due to Esnault and Viehweg \cite{ev1}.

\begin{lemma}\label{lemma:van}
Suppose that $L$ is ample and that the coefficients of $D$ are coprime to $d$.
 If $m+k\not= n$ and $1\le i\le d-1$ then  $H^m(\Omega^k_{Z}(\log
 D)\otimes L^{(-i)})=0$.
\end{lemma}

\begin{proof}
We have
$$H^m(\Omega^k_{Z}(\log D)\otimes L^{(-i)}) \subset H^{m+k}(Y-E,\C)$$
by lemma~\ref{lemma:ev} and item (H3).
Since $Y-E$ is affine, the right side vanishes when $m+k>n$ \cite[1.5]{ev1}.

For the remaining case $m+k<n$, we use hard Lefschetz on $Y$ with respect to the pullback of
$L$. This is compatible with the $\mu_d$-action, and therefore
induces
$$H^m(\Omega^k_{Z}(\log D)\otimes L^{(-i)}) \cong H^{m+j}(\Omega^{k+j}_{Z}(\log D)\otimes L^{(-i)})$$
where $j=n-(m+k)$.
\end{proof}

\begin{cor}\label{cor:van}
If in addition to the above assumptions, $i\not= n$, $H^i(Y,\Q)\cong H^i(Z,\Q)$.
\end{cor}

\section{Beilinson-Hodge}

We say that the Beilinson-Hodge conjecture holds for a smooth variety
$V$ (in degree $j$) if the cycle map on the higher Chow group
$$ CH^j( V,j)\otimes\Q\to Hom_{MHS}(\Q(-j),H^j(V,\Q))$$
is surjective (for the given $j$). See \cite{as, ak, beilinson, jannsen} for more background.

For the next lemma, we use the same notation as in section 1, that $Z$
is smooth projective and $Y\to Z$ is a
$d$-sheeted normal cyclic cover branched over a normal crossing divisor $D$.

\begin{lemma}\label{lemma:BH}
Suppose that the  coefficients of $D$ are coprime to $d$, Beilinson-Hodge
conjecture holds for  $V=Z-D$ in degree $j$, and that
$W_jH^j(V,\Q)=0$. Then
 the Beilinson-Hodge conjecture holds for $U=Y-E$ in  degree $j$.
\end{lemma}

\begin{proof}
We can assume that $j>0$ since otherwise the statement is vacuous.
Since $H^j(Y)$ is pure of weight $j$, $Hom(\Q(-j),H^j(Y))=0$. Therefore
we have an injection
$$Hom_{MHS}(\Q(-j),H^j(U,\Q))\hookrightarrow
Hom_{MHS}(\Q(-j),H^j(U,\Q)/\im H^j(Y,\Q))$$
% We also have
% $$Hom_{MHS}(\Q(-j),H^j(V,\Q))\cong
% Hom_{MHS}(\Q(-j),H^j(V,\Q)/\im H^j(Z,\Q))$$
% because we have assumed that $W_jH^j(Z)= \im H^j(Z)=0$.
  By lemma~\ref{lemma:ev} and (H3), there is an isomorphism
$$H^j(U,\C)= \bigoplus_k  H^{j-k}(\Omega_{Z}^k(\log D) )\oplus
\bigoplus_{k}\bigoplus_{i=1}^{d-1} H^{j-k}(\Omega_{Z}^k(\log D)\otimes
L^{(-i)})$$
$$H^j(Y,\C)= \bigoplus_k  H^{j-k}(\Omega_{Z}^k)\oplus
\bigoplus_{k}\bigoplus_{i=1}^{d-1} H^{j-k}(\Omega_{Z}^k(\log D)\otimes
L^{(-i)})$$
Therefore 
\begin{equation}
  \label{eq:UmodY}
H^j(U,\Q)/\im H^j(Y,\Q) \cong H^j(V,\Q)/\im H^j(Z,\Q)  \cong  H^j(V,\Q)
\end{equation}
where the last isomorphism follows from our assumption that  $W_jH^j(Z)= \im H^j(Z)=0$.
Thus we have a commutative diagram
$$
\xymatrix{
 CH^j(U,j)\otimes \Q\ar^{r}[d] & CH^j(V,j)\otimes \Q\ar@{>>}[d]\ar[l] \\ 
Hom(\Q(-j),H^j(U))\ar@{^{(}->}[d] & Hom(\Q(-j),H^j(V))\ar[ld]_{\sim}\\
 Hom(\Q(-j),H^j(U)/\im H^j(Y)) & 
}
$$
which implies that the map $r$ is necessarily surjective.
\end{proof}

We sketch an alternate proof of  the lemma~\ref{lemma:BH} which is a bit more conceptual.

\begin{proof}[Outline of second proof]
 Consider the diagram 
$$
\xymatrix{
 H^j(Y)\ar[r] & H^j(U)\ar[r] & H_{E}^{j+1}(Y) \ar[r]  & H^{j+1}(Y) \\ 
 H^j(Z)\ar[r]\ar[u] & H^j(V)\ar[r]\ar[u] & H_D^{j+1}(Z)\ar[u]^{\alpha}\ar[r] 
& H^{j+1}(Z)\ar^{\beta}[u]}
$$
where the coefficients are $\Q$.
Since $Y\to Z$ is totally ramified at $D$, we can see that $D$ and $E$
have homotopy equivalent tubular neighbourhoods.
Therefore $\alpha$ is
an isomorphism.  If $X\to Y$ is a resolution of singularities, then
the map $H^{j+1}(Z)\to H^{j+1}(X)$ is injective, because the
normalized pushforward gives a left inverse. Since the factors through
$\beta$, $\beta$ must also be injective.
This implies
\eqref{eq:UmodY} by a diagram chase.
The rest of the proof is identical to the one above.
\end{proof}

We can  extend this lemma to more general branch divisors, but we have
to worry about the effect of  the 
singularities. Given a divisor $D\subset Z$, a log resolution of $(Z,D)$
is a resolution of singularities $p:Z'\to Z$ such that $p^*D$ has simple normal crossings.

\begin{cor}\label{cor:BH}
  Assume all of the  conditions of lemma \ref{lemma:BH}  with the one exception that
 $D$ is only effective. Then Beilinson-Hodge holds in
degree $j$ for $U$  if in addition there exists
a log resolution $p:Z'\to Z$ such that $p^*D$ has all coefficients prime to $d$. 
\end{cor}

 We will say that  $D=\sum a_iD_i$ is of {\em arrangement
type} if  the components are all
 smooth and $D$ is local analytically isomorphic to a hyperplane
arrangement in affine space. This, of course, includes the case where $D\subset \PP^n$ is
itself a hyperplane arrangement. If $p\in D$, let us say that the {\em
  incidence number} at $p$ is the sum of all coefficients $a_j$ for
components $D_j$ containing $p$. In particular, for a reduced divisor,
this is precisely the number of components containing $p$. The set of {\em essential incidence
  numbers} of $D$ is the set of incidence numbers of those $p\in D$ at which
$D$ fails to have normal crossing singularities.

\begin{lemma}\label{lemma:BH2}
Suppose that $Y\to Z$ is a $d$-sheeted cyclic cover branched over an
effective divisor $D=\sum a_iD_i\subset Z$ of arrangement type.
Suppose that  the  coefficients of $D$ and the essential incidence
numbers are coprime to $d$, the Beilinson-Hodge conjecture holds for
$Z-D$ in degree $j$, and that $W_jH^j(Z-D)=0$. Then
 the Beilinson-Hodge  conjecture holds for the preimage of $Z-D$ in the same degree.
\end{lemma}

\begin{proof}
The key point is that we can resolve the singularities of $D$ in
an explicit fashion and keep track of the multiplicities.
For hyperplane arrangements this resolution goes back to De Concini
and Procesi \cite{dp},
although we will follow the simplified presentation of
\cite[\S2.1]{bs}.  Since their procedure is canonical, it applies to
our more general case as well.

Let $D^{nnc}\subset D$ be the non-normal crossing locus. This is the
largest closed subset for which $D\cap (Z-D^{nnc})$ has normal crossings.
We form the
set of  centres
$$S_i = \left\{D_J=\bigcap_{j\in J} D_j\mid D_J\subseteq D^{nnc}, \dim
  D_J=i\right\}$$
for our subsequent blow ups.
We define a sequence of smooth varieties as follows. Take $Z_0=Z$ and
let $Z_1\to Z_0$
be the blow up of $Z_0$ at the union of centres in $S_0$. Let $Z_2$ be the blow up of $Z_1$ at the union (which is a
disjoint union) of the strict
transforms varieties in $S_1$ and so on.
Finally set $Z'=Z_{n-1}$. Then it follows from \cite[\S2.1]{bs}, that
the pullback $D'$ of $D$ to $Z'$ is a divisor with normal crossings.
We claim moreover that the coefficients of $D'$ are coprime to $d$.
This can be checked by induction. Let $D_{(i+1)}$ be the pullback of $D=
D_{(0)}$ to $Z_{i+1}$. This is the sum of the strict transform of $D_{(i)}$
with a sum of exceptional divisors $\sum m_{ij}F_{ij}$. The
coefficients $m_{ij}$ are the multiplicities of $D_{(i)}$ along the
centres of the blow ups, which are precisely the essential incidence numbers.

The result now follows from lemma \ref{lemma:BH}

\end{proof}

\begin{thm}\label{thm:mainBH}
The Beilinson-Hodge conjecture holds for $U$ in the following cases:
\begin{enumerate}
\item[(a)] 
$U$ is the complement of $f(x_0,\ldots, x_n)=0$ in the variety defined
by   $y^d=f(x_0,\ldots, x_n)$, where $f$ is a product of 
  linear forms such that the divisor in $\PP^n$  defined by $f$ satisfies
 the conditions of lemma~\ref{lemma:BH2}.

\item[(b)] $U$  is the complement of $f(x_0,x_1, x_2)=0$ in the variety defined
by   $y^d=f(x_0,x_1, x_2)$, such that $f$ is divisible by a
  linear form and its divisor in $\PP^2$ satisfies
 the conditions of corollary~\ref{cor:BH}.

\end{enumerate}
\end{thm}

\begin{proof}
In case (a), let $f=\prod h_i^{a_i}$ be the factorization as a product of
linear forms. Define $V=\PP^n-\{f=0\}$ as usual.
 Then the classes $\frac{1}{2\pi \sqrt{-1}}d\log h_i$ lie in the image of
the cycle map from $CH^1(V,1)$ essentially by definition. Since the
cycle map is multiplicative, the Beilinson-Hodge conjecture holds for
$V$ because its cohomology is
generated as an algebra by the forms $\frac{1}{2\pi \sqrt{-1}}d\log
h_i$ by Brieskorn \cite[lemma 5]{brieskorn}. This also implies 
that $H^j(V)$ has weight $2j$, so that
$W_j H^j(V)=0$.

For  (b), we use \cite{chat} and the fact that $W_2H^2(V)\subseteq \im
H^2(\C^2)=0$

\end{proof}

\begin{cor}
 In case (a), it suffices that the branch divisor is reduced with
 normal crossings.
\end{cor}

It is worth remarking that Beilinson \cite{beilinson} made a stronger
conjecture that amounts to the surjectivity of the cycle map
$$CH^i(U,j)\otimes \Q\to Hom_{MHS}(\Q(-i), H^{2i-j}(U,\Q))$$
for all $i,j$.  It is now known to be overly optimistic in
general \cite{jannsen}. However,
in the main case of  theorem \ref{thm:mainBH}
(a),  it is 
vacuously true for $0<j<i$ because the above arguments show that 
$$H^i(U)/\im H^i(Y) = H^i(\PP^n-(f))/\im H^i(\PP^n)$$
is a sum of Tate structures $\Q(-i)$. The case of $j=0$ is more subtle.
It is essentially the ordinary Hodge conjecture for $Y$, and
this will be studied in the next section.

\section{Cohomology of toroidal resolutions}

Let $Y$ be defined by $y^d=f$, where $f$  is a product of $d$ distinct
linear forms $h_i$
as in the introduction. We now assume that the  divisor $D=\sum
D_i\subset \PP^n$ defined by $f=0$ is a divisor with normal crossings.    Let $E= \pi^{-1}D$ and $U=
Y-E$. Then $(Y,U)$ is a toroidal embedding in the
sense of \cite{mumford}. This means that about each point of $y\in Y$,
there is   a  neighbourhood $Y_y$ which is
isomorphic to an \'etale open subset of a  toric variety in such a way that
$U\cap Y_y$ maps to the torus.  To see this, we can assume that after a
linear change of coordinates $y$ lies over $[1,0,\ldots,0]\in \PP^n$. Write $f(1,x_1\ldots x_n)= x_1\ldots
x_kg(x_1,\ldots, x_n)$ where $g(0)\not= 0$. Then
$$Y\supset Y_y=\Spec \C[x_1,\ldots,x_n,y]/(y^d-f(1,x_1,\ldots
x_n))\stackrel{F_y}{\to} M_k$$
where the map $F_y$, given by projection, is an open immersion
into  the toric variety 
$$M_k=\Spec \C[x_1,\ldots, x_n,y]/(y^d-x_1\ldots x_k)$$
Later on, we will need to consider the more general case where $D$ is a normal crossing divisor in a
smooth variety, then $(Y,U)$ is still toroidal but the local toric models $M_{(a)}=\Spec \tilde
R_{(a_1,\ldots, a_n, d)}$ need to be chosen as in the
proof of lemma \ref{lemma:toric0} and 
 the corresponding map $F_y$ is only \'etale.
By \cite[p 94]{mumford}, there exists a toroidal
resolution of singularities  $\rho:X\to Y$. In other words, for each $y$,
there is a
commutative diagram
$$
\xymatrix{
 X\ar[d] & \rho^{-1}(Y_y)\ar[d]\ar[r]^{\tilde F_y}\ar[l] & \tilde M_{(a)}\ar^{\tau}[d] \\ 
 Y & Y_y\ar[r]^{F_y}\ar[l] & M_{(a)}
}
$$
where the map $\tau$ is toric, $F_y$ is  \'etale and the right
hand square is cartesian.

\begin{remark}
J. W{\l}odarczyk has pointed out to us that such resolutions are very natural in
the sense that any canonical
resolution algorithm, such as Hironaka's, will yield a toroidal
resolution of $Y$.
\end{remark}

As a prelude to the next theorem, we recall that Grothendieck's amended form of the generalized Hodge
conjecture \cite{groth} says that  sub Hodge
structures of cohomology are induced from subvarieties of expected codimension.
More precisely, recall that the level of a Hodge structure $H\otimes \C=\oplus
H^{pq}$ is the maximum of $|p-q|$ over the  nonzero $H^{pq}$.
The conjecture states that given an irreducible sub Hodge structure
 $H\subset H^i(X,\Q)$ of level $\le i-2k$, there exists a subvariety
$\iota: T\subset X$ of   codimension
$\ge k$ and a desingularization
$\kappa:\tilde T\to T$   such that $H\subset
(\iota\circ \kappa)_*H^{i-2k}(\tilde T,\Q)$.
 This includes the usual Hodge conjecture
which corresponds to the case of $i=2k$.
Our goal is to prove the following. 

\begin{thm}\label{thm:main}
 When $d$ is prime, any irreducible sub Hodge structure of $H^r(X,\Q)$ of
level at most $r-2$ is spanned by an algebraic cycle. In particular,
the generalized Hodge conjecture holds for $X$.
\end{thm}

\begin{remark}
The level and weight have the same parity. So the statement  can be ``strengthened''
by replacing $r-2$ by $r-1$ above.
\end{remark}

This will be deduced from another more general theorem.
 Before stating it, it is convenient to recall the notion of motivic
dimension introduced in \cite{arapura}. Given a smooth projective
variety $Z$, $\mu(Z)\in \N$ roughly measures how much transcendental
cohomology $Z$ has. So $\mu(Z)=0$ holds precisely, when all the
cohomology is generated by algebraic cycles.  In general, $\mu(Z)$ is the smallest  nonnegative integer such that  $H^*(Z)$ is generated by Gysin images of classes of degree at most $\mu(Z)$. The definition can be extended to arbitrary  varieties.
The basic facts we need  are that:
\begin{itemize}
\item $\mu(Z)\le \mu(Z')$ when $Z'\to Z$ is proper and surjective \cite[prop 1.1]{arapura},
\item $\mu(Z)\le \max(\mu(Z'),\mu(Z-Z'))$, when $Z'\subset Z$ is closed \cite[prop 1.1]{arapura}, and
\item
$\mu(Z)\le \mu(Z_s)+\mu(S)$ when $Z\to S$ is a topologically trivial smooth projective map \cite[cor 2.7]{arapura}.
\end{itemize}

\begin{thm}\label{thm:GHC}
  Suppose that $Z$ is a smooth projective variety and that $D\subset
  Z$ is an effective divisor with simple normal crossings such that $\OO_Z(D)=L^d$.
Let  $X$ be a
toroidal resolution of the cyclic $d$-fold cover $Y$ determined by the
data $(D,L)$. We will assume that
  \begin{enumerate}
\item[(a)] The motivic dimensions $\mu(Z)=0$ and $\mu(D_{i_1}\cap\ldots
  D_{i_k})=0$ for all $\{i_1,\ldots, i_k\}$.
\item[(b)] The inequality 
$$\phi(d) h^r(Z,L^{(-d+1)})\ge \dim
H^r_{nt}(Y)=\sum_{k=0}^r\sum_{i=1}^{d-1} h^{r-k}(\Omega_Z^k(\log
D^{(i)})\otimes L^{(-i)})$$
holds, where $\phi$ is the Euler function.
  \end{enumerate}
Then any irreducible sub Hodge structure of $H^r(X,\Q)$ of
level at most $r-2$ is spanned by an algebraic cycle.
\end{thm}

\begin{remark}
The proof will show that  the inequality (b)  is necessarily an equality.
\end{remark}

For the ensuing discussion, let us assume that we are in the more
general situation of theorem \ref{thm:GHC}.
Since $Y$ is an orbifold,  it is   a rational
homology manifold. 
Therefore the natural map $\pi^*:H^i(Y,\Q)\to H^i(X,\Q)$ is injective,
since $\frac{1}{d}\pi_*$ gives a left inverse. Also by Poincar\'e duality, we get a cycle map
on the Chow group of  codimension $k$ cycles $CH^k(Y)\otimes \Q\to H^{2k}(Y,\Q)$.

\begin{prop}\label{prop:HX}
  $H^i(X,\Q)=H^i(Y,\Q)\oplus A^i$ where $A^i$ is generated by
  algebraic cycles.
\end{prop}

\begin{proof}
It is more convenient to work in homology.
Let $F\subset X$ denote the reduced preimage of $E$. We will show that
$\mu(F)=0$ (note that motivic dimension is defined  for singular
varieties as well).  We have a
stratification of $D$ by
\begin{equation}
  \label{eq:D_I}
D_{I}= D_{i_1}\cap\ldots \cap
D_{i_\ell},\quad D_I' = D_I-\bigcup_{j\notin I} D_j
\end{equation}
where $I=\{i_1,\ldots, i_\ell\}$.
Let $F'_I\subseteq F$ be the preimage of $D_F'$, and let $a(I)=(a_{i_1},
\ldots, a_{i_\ell})$. Over  a neighbourhood of $p\in D'_I$, the map $X\to
Y$ is locally isomorphic to the model
$\psi_{a(I)}:\tilde M_{(a(I))}\times \C^{n-\ell}\to M_{(a(I))}\times
\C^{n-\ell}$.
Consequently $ F'_I\to D_I'$ is a Zariski locally trivial  fibration with
fibres  isomorphic to $\psi_{a(I)}^{-1}(0)$. Let $\Psi_{I}$ be a toric resolution of $\psi_{a(I)}^{-1}(0)$.
Then  using the previously stated inequalities, we obtain
$$\mu(F)\le \max_I\mu(F_I')\le\max_I (\mu(D_I')+\mu(\Psi_I))=0$$

Suppose that $\alpha\in H_{i}(X)$ lies in the kernel of $\pi_*$. 
From the diagram
$$
\xymatrix{
 H_{i}(F)\ar[r]\ar[d] & H_{i}(X)\ar[r]\ar[d]^{\pi_*} & H_{i}(X,F)\ar[d]^{\cong} \\ 
 H_{i}(E)\ar[r] & H_{i}(Y)\ar[r] & H_{i}(Y,E)
}
$$
we see that $\alpha$ is the image of  a class in $H_i(F)$ which is
algebraic because $\mu(F)=0$.
Dualizing shows that $H^i(X)$ is generated by $H^i(Y)$ and algebraic cycles.
\end{proof}

\begin{proof}[Proof of theorem \ref{thm:GHC}]
Let $H\subset H^r(X,\Q)$ be an irreducible sub Hodge structure of
level at most $r-2$. By proposition~\ref{prop:HX} and corollary~\ref{cor:ev},
we can decompose $H^r(X,\Q)=H_{nt}^r(Y,\Q)\oplus A$ where $A$ is spanned by
algebraic cycles.  So $H$
is either spanned by an algebraic cycle or it lies in $H_{nt}^r(Y)$.
Assuming the latter, we will show that $H=0$.
Let $H'=\sum \zeta^i H$, where $\zeta\in \mu_d$ is a generator. This is a  $\mu_d$-invariant
Hodge structure containing $H$ and with the same level as $H$.
Thus $H'\otimes \C\subset H_{nt}^{1,r-1}(Y)\oplus\ldots \oplus H_{nt}^{r-1,1}(Y)$
so that $H^r(Y,\OO_Y)\cap H'\otimes \C=0$. Let $H^*= H_{nt}^r(Y)/H'$.  Then
$H^r(Y,\OO_Y)$ injects into $H^*\otimes \C$.
Let $N=\Q[t]/(P(t))$, where 
$$P(t)= \prod_{\gcd(i,d)=1} (t-\zeta^i)$$
is the cyclotomic polynomial. This is the unique irreducible
$\Q[\mu_d]$-module for which $N\otimes \C\supset \C_{\epsilon^{d-1}}$.
It follows that $H^*$ must contain the $m$-fold sum $N^m$, where 
 $m= \dim (H^*\otimes \C)_{\epsilon^{d-1}}$. 
Since
$$(H^*\otimes \C)_{\epsilon^{d-1}}\supseteq H_{\epsilon^{d-1}}^r(Y,\OO_Y)= H^r(L^{-d+1}),$$
we must have $m\ge h^r(L^{-d+1})$. Therefore 
$$\dim H^*\ge \phi(d) h^r(L^{-d+1})\ge\dim H_{nt}^r(Y)$$
by condition (b). Therefore
$H'=0$ as claimed, and the proof is complete.
\end{proof}

We turn to the proof of theorem~\ref{thm:main}. Our main task is to compute the Hodge
numbers.

\begin{lemma}[Hirzebruch]\label{lemma:hirz}
$$\sum_{n=0}^\infty\sum_{k=0}^\infty \chi(\Omega_{\PP^n}^k(i))y^kz^{n}
=\frac{(1+yz)^{i-1}}{(1-z)^{i+1}}$$  
\end{lemma}

\begin{proof}
  This is a special case of  the formula on \cite[p 160]{hirz}. 
\end{proof}

\begin{lemma}\label{lemma:hN}
Let $Y$ be as in theorem~\ref{thm:main}. Then for 
each $k$,
  $$h^{k,n-k}_{nt}(Y) = \binom{d-1}{n+1}= \frac{d-1}{n+1}\binom{d-2}{n}$$
  In particular, this would vanish for $d-1< n+1$.
\end{lemma}

\begin{proof}
From the residue isomorphism \cite{deligne}
$$Gr^W_r\Omega_{\PP^n}^k(\log D)\cong
\bigoplus\Omega_{D_{i_1}\cap\ldots \cap D_{i_r}}^{k-r}$$
Thus we deduce
$$\chi(\Omega_{\PP^n}^k(\log D)(-i)) =\sum_{r=0}^d \binom{d}{r}
\chi(\Omega_{\PP^{n-r}}^{k-r}(-i)) $$
Therefore
$$\sum_{n=0}^\infty\sum_{k=0}^\infty \chi(\Omega_{\PP^n}^k(\log
D)(-i))y^kz^{n} = (1+yz)^d \sum_{n=0}^\infty\sum_{k=0}^\infty
\chi(\Omega_{\PP^n}^k(-i))y^kz^{n}$$
Combining this with lemma~\ref{lemma:hirz} yields
\begin{eqnarray*}
 \sum_{i=1}^{d-1}\sum_{n=0}^\infty\sum_{k=0}^\infty \chi(\Omega_{\PP^n}^k(\log D)(-i))y^kz^{n}
&=&\sum_{i=1}^{d-1} (1-z)^{i-1} (1+yz)^{d-i-1}\\
&=& \frac{(1+yz)^{d-1}-(1-z)^{d-1}}{(1+yz)-(1-z)}\\
&=&\frac{(1+yz)^{d-1}-(1-z)^{d-1}}{z(1+y)}\\
&=& \frac{1}{z[y-(-1)]}\sum_{n=0}^{d-2}\binom{d-1}{n+1}[y^{n+1}-(-1)^{n+1}]z^{n+1}\\
&=&\sum_{n=0}^{d-2} \binom{d-1}{n+1}[y^{n}-y^{n-2}+\ldots \pm 1]z^n
\end{eqnarray*}
By lemmas~\ref{lemma:ev} and~\ref{lemma:van},
$$h^{k,n-k}_{nt}(Y)=(-1)^{n-k-1}\sum_{i=1}^{d-1} \chi(\Omega_{\PP^n}^k(\log
D)(-i))$$
This together with the previous formula implies the lemma.
\end{proof}

\begin{proof}[Proof of theorem \ref{thm:main}]
When $r\not= n$, $H^r(X,\Q)$ is spanned by algebraic cycles by
proposition~\ref{prop:HX} and corollary~\ref{cor:van}. So we only need
consider $r=n$. In this case, we apply theorem~\ref{thm:GHC}.
Condition (a)  of this  theorem is clear. For  (b), we
observe that by the previous lemma, we have
$$(d-1)h^n(\OO_{\PP^n}(-d+1)) =(d-1)\binom{d-2}{n} = \sum_k h_{nt}^{k,n-k}(X)$$
\end{proof}

We can handle some related examples in a similar way:

\begin{cor}\label{cor:branchPxP}
  Let $d$ be prime. The generalized Hodge conjecture holds for a
  toroidal resolution of  the
  cyclic branched cover of $(\PP^1)^n$ given by  
$$y^d=  \prod_{i=1}^n\prod_{j=1}^d (x_i-a_{ij})$$
 where  $a_{i1},a_{i2},\ldots, a_{i,p}$ are distinct for each $i$.
\end{cor}

\begin{proof}
  Let $D\subset (\PP^1)^n$ be the
  divisor given by the union of $x_i-a_{ij}=0$, and let $L=
  \OO(1)\boxtimes \ldots \boxtimes \OO(1)$. 
We have only to check condition (b) of  theorem~\ref{thm:GHC} for $r=n$.
We can compute
  $h^n(L^{-d+1}) = (d-2)^n$ immediately. For the other side, we
define the generating function
$$\chi_{n,i}(y)= \sum_k \chi(\Omega_{(\PP^1)^n}^k(\log D)\otimes L^{-i})y^k$$
Then by K\"unneth's
  formula, we obtain
$$\chi_{n,i}(y) = \chi_{1,i}(y)^n= (1-i+(d-i-1)y)^n$$
We have
$$\sum_{i=1}^{d-1}\sum_k  h^{r-k}(\Omega^k(\log D)\otimes L^{-i})= (-1)^n\sum_{i=1}^{d-1}
\chi_{n,i}(-1) = (d-1)(d-2)^n$$
which implies (b).
\end{proof}

Suppose that $\dim X=
n=2m-1$ is odd. Then we have the Abel-Jacobi map
$$\alpha:CH^m(X)_{hom}\to   J^n(X)= \frac{H^n(X,\C)}{F^m+H^n(X,\Z)}$$
from the homologically trivial part of the Chow group
to  the intermediate Jacobian. Nori \cite{nori} has constructed a
filtration
$$A_0CH^m(X)\subseteq\ldots \subseteq A_{n-m}CH^m(X)=CH^m(X)_{hom}$$
where $A_0$ is the subgroup of cycles algebraically equivalent to $0$.
In general, a cycle lies in $A_r$ if it is induced via a
correspondence from a  homologically
trivial $r$-cycle on another variety.

\begin{cor} Suppose that $X$ is either a variety of the type given in  theorem
\ref{thm:main} or corollary \ref{cor:branchPxP}
with   $n$ is odd. Then $\alpha(A_{n-m-1}CH^m(X))=0$. In particular,
  $\alpha(\eta)=0$ for any cycle $\eta$ algebraically equivalent to $0$.
\end{cor}

\begin{proof}
  The image  $\alpha(A_{n-m-1}CH^m(X))$ lies the subtorus determined by the maximal integral Hodge
  structure contained in $F^1H^n(X)$ \cite{nori}. The theorem implies
  that this Hodge structure is zero.
\end{proof}

This argument also shows that $J^n(X)_{alg}=0$, where
$J^n(X)_{alg}\subset J^n(X)$ is
the maximal abelian subvariety \cite[\S 8.2.1]{voisin}.

\section{Nonabelian covers}

Our goal is to extend the previous estimates to situations where a
possibly nonabelian finite group $G$ acts on a variety. This will
apply  in particular to $G$-covers.
Let $\hat G$ be the set of characters of irreducible $\C[G]$-modules, and $1\in \hat G$
the character of the trivial module.
Given a character $\chi$ of an irreducible $\C[G]$-module, let 
$$e_\chi = \frac{\chi(1)}{|G|}\sum_{g\in G}\bar\chi(g)g$$
denote the corresponding central idempotent \cite[thm 33.8]{cr}. 
This determines the $\chi$-isotypic submodule of a
$\C[G]$-module $M$ by  $M_\chi= e_\chi M$. Let $M_{nt}=\sum M_\chi$, as $\chi$ ranges
over the nontrivial characters.

We introduce an  invariant that will measure the difference between the complex
and rational representation theory.
Given a finite dimensional $\C[G]$-module  $M$, we define the {\em
  rational   span} as the minimal (with respect to dimension)  $\Q[G]$-module $M'$
such that $M'\otimes \C\supseteq M$. 
Of course, the span  is only an isomorphism
class, but its character is well defined, as are the numbers
$$\sigma(M)=\dim_\Q M',\quad \Phi(M)= \frac{\sigma(M)}{\dim_\C M}$$
A character will be
called rational, if the associated $\C[G]$-module is realizable over
$\Q$. Given a character $\sum_{\chi\in \hat G} n_\chi\chi$, the
character of its rational span can be characterized as the rational
character $\sum r_\chi\chi$ with $r_\chi\ge n_\chi$ such that $\sum r_\chi$ is minimal.
 We let $\Q(\chi)$ denote the extension of $\Q$
obtained by adjoining the values  $\chi(g)$. The Schur index $m(\chi)$
is the degree of the smallest extension of $\Q(\chi)$ over
which $M$ can be realized, cf. \cite[41.4]{cr}. The Galois group
$Gal(\bar\Q/\Q)$ acts on $\hat G$; $orb(\hat G)$ will denote the set
of orbits. The orbit $orb(\chi)$ of a given $\chi$ is the set of
Galois conjugates  $\gamma\chi$ with $\gamma\in Gal(\Q(\chi)/\Q)$, and these are all distinct.

\begin{lemma}\-
  \begin{enumerate}
  \item[(a)] If $\chi\in \hat G$,  $\Phi(\chi)$ is the product of $m(\chi)$ and the degree of
  $\Q(\chi)$ over $\Q$. In particular, it is
 an integer.
\item[(b)] For a non irreducible character $\xi=\sum_{\chi\in \hat G}
  n_\chi \chi$,
$$\sigma(\xi) =\sum_{\Gamma\in orb(\hat G)} \max_{\chi\in \Gamma}
\left\lceil \frac{n_\chi}{m(\chi)}\right\rceil\sigma(\chi)$$ 
where $\lceil\, \rceil$ is the round up or ceiling function.
  \end{enumerate}
\end{lemma}

\begin{proof}
  (a) is an immediate consequence of \cite[thm 70.15]{cr} which
implies that the character of the span of $\chi\in \hat G$ is
$\sum_{\chi'\in orb(\chi)} m(\chi)\chi'$.
This also implies (b) by the above remarks.
\end{proof}

\begin{remark}
When $\chi\in \hat G$, $m(\chi)$ and $\sigma(\chi)$ are Galois
invariant, so we can write these as functions of the  
orbit. So the formula (b) can be simplified ever so slightly to
$$\sigma(\xi) =\sum_{\Gamma\in orb(\hat G)} \sigma(\Gamma)\max_{\chi\in \Gamma}
\left\lceil \frac{n_\chi}{m(\Gamma)}\right\rceil$$ 
\end{remark}

Armed with this formula, and standard facts from \cite[\S 28, \S 70]{cr},
we can compute a number of examples:
\begin{enumerate}
\item[(a)]  If $G=\Z/d\Z$, then $\Phi(\epsilon^i)=\phi(d/\gcd(i,d))$.
\item[(b)] If $G=S_N$ is the symmetric group, $\Phi(\chi)=1$ for all $\chi$.
\item[(c)] If $G=\{\pm 1,\pm i, \pm j,\pm k\}$ is the quaternion group, and
  $\chi$ the character of the  unique $2$ dimensional irreducible complex representation, $\Phi(\chi)=m(\chi)=2$.
\end{enumerate}

We come to  the key estimate.  We first note that if $M$ is
$\Q[G]$-module, then so is $M_{nt}\cong M/M_1$.

\begin{prop}\label{prop:nonab}
  Suppose that $G$ is a finite group of automorphisms of a rational
  pure  effective Hodge structure $H$ of weight $i$. The dimension of any sub Hodge
  structure $H'\subset H_{nt}$ of
level less than $i-2k$ is bounded above by the  difference 
$$\dim H_{nt}-\sigma(H_{nt}^{0i}\oplus\ldots\oplus
H_{nt}^{k,i-k}\oplus\overline{H_{nt}^{0i}\oplus\ldots \oplus H_{nt}^{k,i-k}})$$

\end{prop}

\begin{proof}
Let $T= H_{nt}^{0i}\oplus\ldots\oplus
H_{nt}^{k,i-k}\oplus\overline{H_{nt}^{0i}\oplus\ldots \oplus H_{nt}^{k,i-k}}$.
Given a sub Hodge structure $H'$ of 
level $<i-2k$, by replacing it with $\sum gH'$, we can assume without loss of generality that  it is
$G$-invariant.  By the level assumption
$T\cap (H'\otimes \C)=0$.
Thus $H_{nt}/H'$  is a $\Q[G]$-module
containing $T$ after extending
scalars.  Therefore $H/H'$ contains the rational span of $T$.
\end{proof}

Putting everything together yields:

\begin{cor}\label{cor:nonab}
If $G$ is a finite group acting on a projective orbifold $Y$,  the
dimension of a sub Hodge structure of $H_{nt}^i(Y)$ of 
level less than $i-2k$ is bounded above by the  difference 
\begin{equation}
  \label{eq:nonabBnd}
\dim H_{nt}^i(Y)- \sum_{\Gamma\in orb(\hat
  G-1)}\sum_{\substack{p+q=i\\p-q\ge k}} |\Gamma|m(\Gamma) \max_{\chi\in\Gamma}
\left\lceil \frac{h^{p,q}_\chi(Y) }{m(\Gamma)\chi(1)}\right\rceil 
\end{equation}
In particular when $i$ is even, we get a bound on the dimension  the space of Hodge
cycles  by applying this with $k=i/2-1$.
\end{cor}

The main remaining issue is whether we can actually compute this bound.
We will work this out for covers. We will fix the following notation
for the remainder of this section.
Let $Z$ be an $n$ dimensional smooth projective
variety. Let $D\subset Z$ be a reduced effective divisor with simple normal
crossings and let $V=Z-D$.   
Suppose that $\rho:\pi_1(V)\to G$ is
surjective homomorphism onto a finite group. We can construct the
associated \'etale cover $U\to V$. Let $\pi:Y\to Z$ be the normalization of $Z$
in the function field of $\C(U)$, and let $E=\pi^{-1} D$. We refer to the triple
$(Z,D,\rho)$ as the {\em branching data} for $Y$.

We first analyze the local picture.
 We can cover $Z$ by coordinate polydisks $\Delta_i$
so that $D$ is given by $x_1\ldots x_{k_i}=0$. 
Let us fix one of these, and suppress the subscript $i$ below.
Then the  fundamental group
$\pi_1(\Delta-D)=\Z^k$ with generators corresponding to loops around
the coordinate hyperplanes.  
Thus the preimage of $\Delta-D$ in $Y$ is
given by a disjoint union of connected abelian covers of $\Delta-D$.
We can describe these components explicitly.

\begin{lemma}\label{lemma:abelian}
A normal connected abelian cover $\tilde\Delta\to \Delta$ is an open set (in the classical
topology) of a normal affine toric variety 
  with finite quotient singularities.  The projection to $\Delta$ is   flat.
\end{lemma}

\begin{proof}
  The cover $\Delta$ is determined by a subgroup $\Gamma\subset \pi_1(\Delta-D)=\Z^k$ of
  finite index. By elementary divisor theory, a basis for $\Gamma$ is
  given the columns of a diagonal matrix with positive entries
  $d_i$. Thus $\Delta$ is equivalent to an neighbourhood of the
  normalization $\Spec \tilde R$  of the variety $\Spec R$ defined 
\begin{equation}
  \label{eq:toric}
y_i^{d_i}=\prod_{j\in J_i} x_{j}^{a_{ij}}, i=1\ldots k
\end{equation}
where the sets $J_i\subset\{1,\ldots, k\}$ are disjoint. These are
tensor products of the rings considered earlier in the proof of
lemma~\ref{lemma:toric0}.
The results proved there shows this is a toric variety with quotient singularities.
Also the projection is flat because $\tilde R$ is a free module over
$\C[x_1,\ldots, x_n]$, cf \cite[\S 3]{ev}.
\end{proof}

\begin{cor}\label{cor::nonabelian}
  $Y$ is a toroidal  orbifold. Moreover, the map $Y\to Z$ is flat.
\end{cor}

It follows that we can construct a toroidal resolution $X$ of $Y$ as
before. Fix one such.
We will say that the {\em weak Lefschetz property} holds for $\pi$ if
the map $H^i(Z,\Q)\to H^i(Y,\Q)$ is an isomorphism for $i\not= n$.
For example, we saw that cyclic covers totally ramified along $D$ have
this property. We now come to the main result.

\begin{thm}\label{thm:nonab}
  With the  above notations:
  \begin{enumerate}
  \item[(a)] The dimension of any  sub Hodge structure of $H_{nt}^i(Y,\Q)$ of
level at most $i-2k-2$ is bounded above by the expression given in
\eqref{eq:nonabBnd} of corollary~\ref{cor:nonab}.

\item[(b)] If this bound is  zero and if the assumptions of   theorem \ref{thm:GHC} (a)
hold, then the generalized Hodge conjecture is valid for $X$.

\item[(c)]If the weak Lefschetz property holds, then the bound
\eqref{eq:nonabBnd}
can be computed by an explicit formula involving only the branching data.
  \end{enumerate}
\end{thm}

Most of this follows from we have said previously.
The only thing that we need to explain is the last statement (c).
This will require a bit of preparation.

The  sheaf $\V=\pi_*\OO_Y$ is a vector bundle with a $G$-action.
Set $\V_\chi = e_\chi \V$ as usual. Then we can
decompose
$$\V = \bigoplus_{\chi\in \hat G} \V_\chi=\OO_Z\oplus \V_{nt}$$
We see that 
$$H_\chi^i(Y,\OO_Y) = e_\chi H^i(Z,\V)=  H^i(Z,\V_\chi)$$
$\V$ carries an integrable  logarithmic  connection
$$\nabla:\V\to \V\otimes \Omega_Z^1(\log D)$$
which is none other than
the canonical extension of the Gauss-Manin connection given by the direct image of $d:\OO_Y\to
\tilde \Omega_Y^1(\log E)$ \cite{deligne-diff}. The restriction of 
$\nabla$ to $U$ can also be characterized by the fact that the
underlying local
system $\ker \nabla|_U$ is given by $\pi_*\C|_U$. 
All of this is compatible with the $G$-action. The restriction of 
$\nabla$ to $\V_\chi$ is just $e_\chi\nabla$.

The fact that $\nabla$ is a canonical extension, with finite monodromy, means that
over a polydisk, we can choose local coordinates and a local frame
$\{e_j\}$ for $\V_\chi$ so that
$$\nabla = d +\sum R_i
\frac{dx_i}{x_i}
$$
with
$$R_i=
\begin{pmatrix}
  r_{i1} &0 &0\\
 0& r_{i2}&\\
0&0&\ddots
\end{pmatrix}
$$
where $r_{ij}\in [0,1)\cap \Q$. 
The above matrices are  determined from monodromy by taking normalized
logarithms \cite[p 54]{deligne-diff}. More precisely, we can choose a
matrix representation $T$ isomorphic to the regular
representation $\C[G]$, such that
$$ T(\rho(\gamma_i)) =\exp(-2\pi \sqrt{-1}R_i)$$

We can compute the Hodge numbers  in terms of logarithmic differentials.
Following Timmerscheidt \cite{timmer}, we
have a subbundle
$$W_0(\Omega_Z^k(\log D)\otimes \V)\subseteq \Omega_Z^k(\log D)\otimes
\V$$
locally spanned by wedge products of
$$
\begin{cases}
e_j\otimes dz_i & \text{if $r_{ij}=0$}\\
 e_j\otimes \frac{dz_i}{z_i} &  \text{otherwise}
\end{cases}
$$
for a frame chosen with diagonal connection matrix
as above.  The utility of this construction for us stems from the following:

\begin{lemma}
We have
 $$\pi_*\tilde \Omega_Y^k(\log E)= \Omega_Z^k(\log D)\otimes \V$$
and
  $$\pi_*\tilde \Omega_Y^k= W_0(\Omega_Z^k(\log D)\otimes \V)$$
\end{lemma}

\begin{proof}
 Since it suffices to check this locally away from the codimension two
 set $D_{sing}$, we can reduce almost immediately to  a polydisk centered
 around a component of $D$, whence to the cyclic case.
In this case, the arguments can be found already in \cite[\S
3]{ev}. To elaborate a bit more. The first equality follows from the projection formula
and the fact that $\pi^*\Omega^k_Z(\log D)=\tilde \Omega_Y^k(\log E)$.
For the second equality, we will be content to work out the basic example
of $y^d= x_1$  with the local frame  $1,y,\ldots, y^{d-1}$. We have
\begin{equation}
  \label{eq:res}
dy^i= \frac{(i-1)}{d}y^i \frac{dx_1}{x_1}.  
\end{equation}
Therefore, we see that both sides are spanned by
products of
$$dx_1,dx_2,\ldots; y\frac{dx_1}{x_1}, ydx_2\ldots; \ldots;
y^{d-1}\frac{dx_1}{x_1}, y^{d-1}dx_2,\ldots$$
\end{proof}

\begin{cor}
$$H^{p,q}_{\chi}(Y)\cong H^{q}(W_0(\Omega_Z^p(\log
D)\otimes \V_{\bar \chi}))$$
\end{cor}

The forms  lying in $W_0$ should be thought of as nonsingular, in
analogy with the usual case. The submodule
$W_\ell(\Omega_Z^k(\log D)\otimes \V) \subseteq \Omega_Z^k(\log D)\otimes \V$ is defined by locally allowing
sums of wedge products of at most $\ell$ forms singular forms.
There is a Poincar\'e residue isomorphism  \cite{timmer}
\begin{equation}
  \label{eq:res}
  Gr^W_\ell(\Omega_Z^k(\log D)\otimes \V)\cong \bigoplus_{|I|=\ell}
W_0(\Omega_{D_I}^{k-\ell}(\log D_I'')\otimes \V_I)
\end{equation}
where $D_I,D_I'$ are as in \eqref{eq:D_I}, $D_I''=D_I-D_I'$ and
$\V_I\subset \V|_{D_I}$ is the  subbundle  corresponding to the
local system $j_*(\ker \nabla|_U)|_{D_I'}$. In other words, the extension of $\V_I$ to a
tubular neighbourhood corresponds to the maximal
sub local system of  $\ker \nabla|_U$ with trivial monodromy around
components of $D_i,i\in I$.

\begin{lemma}\label{lemma:K0}
  The class of $W_0(\Omega_Z^k(\log D)\otimes \V)$ in the Grothendieck
  group $K_0(Z)$ is
$$\sum_I (-1)^{|I|} \Omega_{D_I}^{k-|I|}(\log D_I'')\otimes \V_I$$
A similar formula holds for each $\V_\chi$.
\end{lemma}

\begin{proof}
  By \eqref{eq:res}, we have 
  \begin{equation*}
\Omega_Z^k(\log D)\otimes \V = \sum_I W_0( \Omega_{D_I}^{k-|I|}(\log
D_I'')\otimes \V_I)    
  \end{equation*}
Then the lemma follows by the M\"obius inversion
formula \cite[prop 2]{rota}. Or more directly,
we can solve
$$ W_0(\Omega_Z^k(\log D)\otimes \V) = \Omega_Z^k(\log D)\otimes \V
-\sum_{I\not=\emptyset} W_0( \Omega_{D_I}^{k-|I|}(\log
D_I'')\otimes \V_I)$$
We can assume that the lemma holds for each  proper $D_I$ by
induction. Substituting the resulting expressions into the one above
and simplifying, yields the lemma.
\end{proof}

We are now ready to finish the proof of the main theorem.

\begin{proof}[Proof of theorem \ref{thm:nonab} (c)]

It will be convenient to fix a choice of loops
 $\gamma_j\in \pi_1(V)$ around each component of $D_j$.
It will be clear that the formulas will ultimately depend only on their
conjugacy classes which are completely determined by the branching data.
Since $G$ acts on the sheaf
$\pi_*\C_Y$, we can decompose it as $\C_Z\oplus (\pi_*\C_Y)_{nt}$.
So we have $\dim H^i_{nt}(Y)= \dim H^i(Y)-\dim H^i(Z)$ which is zero
unless $i=n$. Thus we conclude that \eqref{eq:nonabBnd} is trivial for
$i\not=n$ and that
$$(-1)^n\dim H^n_{nt}(Y) = e(Y)-e(Z)$$
where $e$ denotes the topological Euler characteristic. The right side is
easily computed as
\begin{multline*}
d(e(Z-D)) + \sum_{|J|>0,D_J\not=\emptyset} d_{J}e(D_{J}') - e(Z)  \\
=  |G|(e(Z-D)) + \sum_{|J|>0,D_J\not=\emptyset} \frac{|G|}{|G(J)|}e(D_J')- e(Z)  
\end{multline*}
where we write $d_{J}$ for the number of sheets over $D_{J}'$, and $|G(J)|$
for the  order of the stabilizer of a component of $\pi^{-1}D_J'$.
Then we see that $|G(J)|$ is
the dimension of the intersection of kernels of the action of
$\rho(\gamma_j),j\in J$ on the
regular representation $\C[G]$. Thus we 
 have our desired formula for the first part $\dim H^n_{nt}(Y)$ in
 terms of branching data.

To finish the proof, it suffices to give formulas for the Hodge
numbers $h_\chi^{p,i-p}(Y) $ with $\chi\not=1$.
The group 
$$H_\chi^{p,i-p}(Y) = H^{i-p}(W_0(\Omega_Z^p(\log D)\otimes V_\chi))$$
 is a summand of  $H^i_{nt}(Y)$, so it is zero when $i\not= n$.
Therefore $h_\chi^{p,i-p}(Y) $ is a holomorphic Euler characteristic up to
sign. When combined with Hirzebruch-Riemann-Roch \cite{hirz}, we obtain
$$(-1)^{n-p}h_\chi^{p,n-p} = \int_Z ch(W_0(\Omega_Z^p(\log D)\otimes V_\chi)) td(Z)$$
So the only thing remaining is to evaluate the Chern character in
terms of the branching data.

By applying a result of Esnault and Verdier \cite[appendix B]{ev1}, we obtain 
\begin{equation}\label{eq:chV}
  \begin{split}
    ch(V_\chi) &= 
\sum_p \sum_{m_1+m_2\ldots=p} \frac{(-1)^p}{p!}
\binom{p}{m_1,m_2\ldots} tr((e_\chi R_1)^{m_1}(e_\chi
R_2)^{m_2}\ldots) [D_1]^{m_1}[D_2]^{m_2}\ldots\\
&= \sum_p \sum_{m_1+m_2\ldots=p} \frac{(-1)^p}{p!}
\binom{p}{m_1,m_2\ldots} tr(e_\chi R_1^{m_1}
R_2^{m_2}\ldots) [D_1]^{m_1}[D_2]^{m_2}\ldots
\end{split}
\end{equation}
Since this involves only the branching data, we have our desired
formula for $h_\chi^{0,n}$. For the other $p$'s there is one extra
step.  By lemma~\ref{lemma:K0} and the fact $ch$ is ring homomorphism, we have
\begin{equation}
  \label{eq:chOm}
  ch(W_0(\Omega_Z^k(\log D)\otimes \V_\chi) =
\sum (-1)^{|I|} ch(\Omega_{D_I}^{k-|I|}(\log D_I'')) ch(\V_I)
\end{equation}
We have a  formula for $ ch(V_{I,\chi})$, similar to \eqref{eq:chV}, involving restrictions
of the residues to $D_I$. So that \eqref{eq:chOm}
 can be expanded to a formula of the desired type.

\end{proof}

%%%

\end{document}